\documentclass{amsart}
\usepackage{amssymb}

\usepackage{graphicx}

\newtheorem{theorem}{Theorem}[section]
\newtheorem{lemma}[theorem]{Lemma}
\newtheorem{corollary}[theorem]{Corollary}
\newtheorem{proposition}[theorem]{Proposition}
\newtheorem{claim}[theorem]{Claim}
\newtheorem{question}[theorem]{Question}

\theoremstyle{definition}

\newtheorem{example}[theorem]{Example}

\theoremstyle{remark}

\numberwithin{equation}{section}

\usepackage{caption} 
\begin{document}

\title[Hyperbolic L-space knots and their Upsilon invariants]{Hyperbolic L-space knots and their Upsilon invariants}


\author{Masakazu Teragaito}
\address{International Institute for Sustainability with Knotted Chiral Meta Matter (WPI-SKCM$^2$), Hiroshima University, 1-3-1, Kagamiyama, Higashi-hiroshima 7398526, Japan}
\email{teragai@hiroshima-u.ac.jp}
\thanks{The author  has been supported by
JSPS KAKENHI Grant Number 20K03587. }


\date{}


\commby{}

\begin{abstract}
For a knot in the $3$--sphere, the Upsilon invariant is a piecewise linear function
defined on the interval $[0,2]$. 
For an L--space knot, the Upsilon invariant is determined only by the Alexander polynomial of the knot.
We exhibit infinitely many pairs of hyperbolic L--space knots
such that two knots of each pair have distinct Alexander polynomials, so they are not concordant, but share the same Upsilon invariant.
Conversely, we examine the restorability of the Alexander polynomial of an L--space knot
from the Upsilon invariant through the Legendre--Fenchel transformation.
\end{abstract}

\maketitle


\section{Introduction}

For a knot $K$ in the $3$--sphere $S^3$, Ozsv\'{a}th, Stipsicz and Szab\'{o} \cite{OSS} defined
the Upsilon invariant $\Upsilon_K(t)$, which is a piecewise linear real-valued function defined on the interval $[0,2]$.
This invariant is additive under connected sum of knots, and the sign changes for the mirror image of a knot.
Also, it gives a lower bound for the genus, the concordance genus and the four genus.
Although it is originally defined through some modified knot Floer complex, 
Livingston \cite{L} later gives an alternative interpretation on the full knot Floer complex $\mathrm{CFK}^\infty$.

As the most important feature, the Upsilon invariant is a concordance invariant, so
it is obviously not strong to distinguish knots, although it has been used to establish various powerful results
about independent elements in the knot concordance group \cite{FPR, Ho,OSS,W1}.
For a smoothly slice knot, the Upsilon invariant is the zero function.
It depends only on the signature
for an alternating knot or a quasi-alternating knot \cite{OSS}.
Also, it is determined by the $\tau$--invariant for concordance genus one knots \cite{FPR}.

In this paper, we concentrate on L--space knots, which are recognized to form an important class of knots in recent research.
A knot is called an \textit{L--space knot\/} if it admits a positive surgery yielding
an L--space.
Positive torus knots are typical examples of L--space knots.
Note that any non-trivial L--space knot is prime \cite{K0} and non-slice \cite{OS}.
For an L--space knot, the Upsilon invariant is determined only by the Alexander polynomial \cite[Theorem 6.2]{OSS}.

There is another interesting route to lead to the Upsilon invariant of an L--space knot.
The Alexander polynomial gives the formal semigroup \cite{W}, in turn, the gap function \cite{BL0}.
These notions have the same information as the Alexander polynomial.
Then the Upsilon invariant is obtained as the Legendre--Fenchel transform of the gap function \cite{BH}.

In general, the gap function for an L--space knot is not convex, so
the further Legendre--Fenchel transformation on the Upsilon invariant 
does not return the original gap function.
Thus there is a possibility that
distinct gap functions, equivalently Alexander polynomials, correspond
to the same Upsilon invariant.
In other words, it is expected to exist non-concordant L--space knots with the same Upsilon invariant.
We remark that the Alexander polynomial is a concordance invariant for L--space knots \cite{K0}.
Our main result shows that this is possible among hyperbolic L--space knots.

\begin{theorem}\label{thm:main}
There exist infinitely many pairs of hyperbolic L--space knots $K_1$ and $K_2$ such that
they have distinct Alexander polynomials but share the same non-zero Upsilon invariant.
\end{theorem}

Thus two hyperbolic L--space knots in our pair are not concordant.
In the literature, there are plenty of examples of non-concordant knots sharing the same Upsilon invariant
\cite{Al, FPR, Ho2,W1, W2,W3,X}.
However, they use either connected sums of torus knots or satellite knots, which
are not hyperbolic.

Since the Upsilon invariant is determined only by the Alexander polynomial
for an L--space knot,
any pair of L--space knots sharing the same Alexander polynomial have
the same Upsilon invariant.
For example, the hyperbolic L--space knot \texttt{t09847} in the SnapPy census has the same Alexander polynomial
as the $(2,7)$--cable of $T(2,5)$, which is an L--space knot.
There are infinitely many such pairs consisting of a hyperbolic L--space knot
and an iterated torus L--space knot (found in \cite{BK}).

However, we checked Dunfield's list of $632$ hyperbolic L--space knots (\cite{A,BK}), and
confirmed that there is no duplication among their Alexander polynomials and that there is no one
sharing  the same Alexander polynomial as a torus knot.
This leads us to pose a question.

\begin{question}
\begin{enumerate}
\item[(1)]
Do there exist hyperbolic L--space knots which have the same Alexander polynomial,
equivalently, which are concordant?
\item[(2)]
Does there exist a hyperbolic L--space knot which is concordant to a torus knot?
\end{enumerate}
\end{question}

In general, it is rare that
the Alexander polynomial of an L--space knot is restorable from the Upsilon invariant.
The reason is the fact that the gap function, which has the same information as the Alexander polynomial,
is not convex, and the Upsilon invariant depends only on the convex hull of
the gap function.
In fact, our knots in Theorem \ref{thm:main} are designed so that
they have distinct Alexander polynomials, but their gap functions share the same convex hull, so
the same Upsilon invariant.

On the other hand,
there is a chance that the gap function is restorable from its convex hull.
This means that the Alexander polynomial is also restorable from the Upsilon invariant
through the Legendre--Fenchel transformation.
We can give infinitely many such gap functions, equivalently Alexander polynomials,
but there lies a hard question, called
a geography question,  whether such gap function can be realized by an L-space knot or not.

In this paper, we can give only two hyperbolic L--space knots whose Alexander polynomials
are restorable from the Upsilon invariants.

\begin{theorem}\label{thm:restorable}
Let $K$ be the hyperbolic L--space knot \texttt{t09847} or \texttt{v2871} in the SnapPy census.
Then the Alexander polynomial $\Delta_K(t)$  of $K$ is restorable from the Upsilon invariant $\Upsilon_K(t)$.
That is,  the equation $\Upsilon_K(t)=\Upsilon_{K'}(t)$ implies $\Delta_K(t)=\Delta_{K'}(t)$ (up to units) for any other L--space knot $K'$.
\end{theorem}

In Section \ref{sec:knots}, we give a pair of knots $K_1$ and $K_2$, which yields
an infinite family of pairs of L--space knots.
In Section \ref{sec:alex}, we calculate their Alexander polynomials
and the formal semigroups, which are sufficient to prove that
the knots are hyperbolic.
Section \ref{sec:upsilon} gives
the gap functions and their convex hulls, and
confirm that they correspond to the same Upsilon invariant.
Section \ref{sec:mont} shows that
the knots admit L--space surgery through the Montesinos trick, which completes the proof of Theorem \ref{thm:main}.
In the last section, we investigate the restorability of Alexander polynomial from the Upsilon invariant, and
prove Theorem \ref{thm:restorable}.

\section{The pairs of hyperbolic L-space knots}\label{sec:knots}

For any integer integer $n\ge 1$, 
the surgery diagrams illustrated in Figure \ref{fig:knots} define
our knots $K_1$ and $K_2$, where the surgery coefficient on $C_1$ is $-1/n$ and
that on $C_2$ is $-1/2$.
The images of $K$ after these surgeries in (1) and (2) of Figure \ref{fig:knots} give $K_1$ and $K_2$, respectively.
(The link with orientations is placed in a strongly invertible position, and
the axis is depicted there for later use.)

\begin{figure}[htpb]
\begin{center}
\includegraphics[scale=0.34]{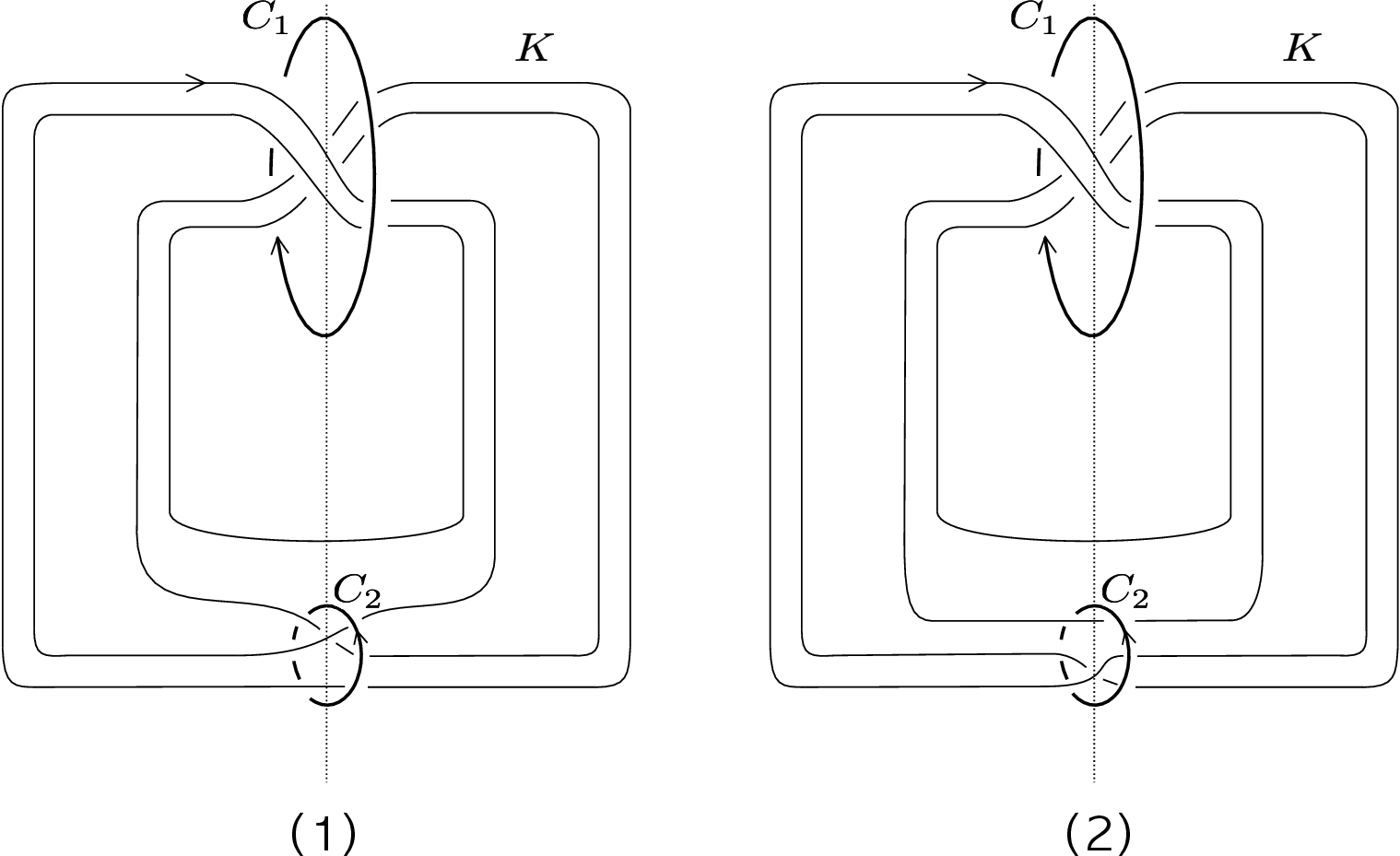}
\caption{The knots $K_1$ and $K_2$ are the images of $K$ after performing $(-1/n)$--surgery on $C_1$ and $(-1/2)$--surgery on $C_2$.}\label{fig:knots}
\end{center}
\end{figure}

Hence, our knots are the closures of $4$--braids
\[
(\sigma_2\sigma_1\sigma_3 \sigma_2) (\sigma_1\sigma_2\sigma_3)^{4n}  \sigma_2^{-1} (\sigma_2\sigma_3)^6\quad  \text{and}\quad
(\sigma_2\sigma_1\sigma_3 \sigma_2) (\sigma_1\sigma_2\sigma_3)^{4n}  \sigma_3^{-1} (\sigma_2\sigma_3)^6,
\]
where
$\sigma_i$ is the standard generator of the $4$--strand braid group.
 When $n=1$, $K_1$ is \texttt{m240}, and $K_2$ is \texttt{t10496} in the SnapPy census \cite{CDG}.

Theorem \ref{thm:main} immediately follows from the next.

\begin{theorem}\label{thm:knots}
For each integer $n\ge 1$,
the knots $K_1$ and $K_2$ defined above satisfy the following.
\begin{itemize}
\item[(1)]
They are hyperbolic.
\item[(2)]
$(16n+21)$--surgery on $K_1$ and
$(16n+20)$--surgery on $K_2$ yield L--spaces.
\item[(3)]
Their Alexander polynomials are distinct.
\item[(4)]
They share the same Upsilon invariant.
\end{itemize}
\end{theorem}

\begin{proof}
This follows from Lemmas \ref{lem:hyp}, \ref{lem:K1mont}, \ref{lem:K2mont},
Theorems \ref{thm:alex1}, \ref{thm:alex2}, and Corollary \ref{cor:upsilon}.
(To see that the Alexander polynomials of $K_1$ and $K_2$ are distinct,
it may be easier to compare their formal semigroups. 
From Propositions \ref{prop:formal1} and \ref{prop:formal2}, we have that
$4n+7\in \mathcal{S}_{K_1}$, but $4n+7\not\in \mathcal{S}_{K_2}$.)
\end{proof}

Each diagram in Figure \ref{fig:knots} has a single negative crossing, but
it can be cancelled obviously with some positive crossing.
Hence both knots are represented as the closures of positive braids, which implies that
they are fibered \cite{S}.
Then it is straightforward to calculate their genera $g(K_i)$, and
we see that $g(K_1)=g(K_2)=6n+6$.

Also, if once we know that $K_i$ is an L--space knot, then
$r$--surgery on $K_i$ gives an L--space if and only if $r \ge 2g(K_i)-1=12n+11$ by \cite{He,OS3}.
Our choices of surgery coefficients in Theorem \ref{thm:knots}(2) come from
the manageability in the process of the Montesinos trick in Section \ref{sec:mont}.

\section{Alexander polynomials}\label{sec:alex}

We calculate the Alexander polynomials of $K_1$ and $K_2$.
Since $K_1$ and $K_2$ are obtained from $K$ by performing some surgeries on $C_1$ and $C_2$,
we mimic the technique of \cite{BK}.

\begin{theorem}\label{thm:alex1}
The Alexander polynomial of $K_1$ is  given as
\[
\begin{split}
\Delta_{K_1}(t) &=
\sum_{i=0}^n (t^{8n+12 +4i}-t^{8n+11+4i})+(t^{8n+9}-t^{8n+8})+\sum_{i=0}^n(t^{4n+6+4i}-t^{4n+4+4i})\\
&\quad +(t^{4n+3}-t^{4n+1})+\sum_{i=0}^{n-1}(t^{4+4i}-t^{1+4i})+1.
\end{split}
\]
\end{theorem}

\begin{proof}
Let $L=K\cup C_1\cup C_2$ be the oriented link illustrated in Figure \ref{fig:knots}(1).
Its multivariable Alexander polynomial is 
\[
\begin{split}
\Delta_L(x,y,z) &=x^6y^3z^2+x^5y^2z-x^3y^3z^2+x^3y^2z^2-x^3y^2z\\
&\quad -x^2y^2z^2+x^4y+x^3yz-x^3y+x^3-xyz-1,
\end{split}
\]
where the variables $x,y,z$ correspond to the (oriented) meridians of $K$, $C_1$, $C_2$, respectively.
(We used \cite{CDG, Ko} for the calculation.)

Performing $(-1/n)$--surgery on $C_1$ and $(-1/2)$--surgery on $C_2$ changes
the link $K\cup C_1\cup C_2$ to $K_1\cup C_1^n\cup C_2^n$.
These two links have homeomorphic exteriors.
Hence the induced isomorphism of the homeomorphism on their homology groups relates
the Alexander polynomials of two links \cite{F,M}.

Let $\mu_K$, $\mu_{C_1}$ and $\mu_{C_2}$ be the homology classes of meridians
of $K$, $C_1$, $C_2$, respectively.
We assume that each meridian has linking number one with the corresponding component.
Furthermore, let $\lambda_K$, $\lambda_{C_1}$ and $\lambda_{C_2}$ be the homology classes of their oriented longitudes.
We see that $\lambda_{C_1}=4\mu_K$ and $\lambda_{C_2}=3 \mu_K$.

Let $\mu_{K_1}$, $\mu_{C_1^n}$ and $\mu_{C_2^n}$ be the homology classes of meridians 
of $K_1$, $C_1^n$ and $C_2^n$.
Then we have that
$\mu_{K_1}=\mu_K$, $\mu_{C_1^n}=-\mu_{C_1}+n \lambda_{C_1}$, $\mu_{C_2^n}=-\mu_{C_2}+2\lambda_{C_2}$.
Hence
\[
\mu_{K}=\mu_{K_1}, \quad \mu_{C_1}=-\mu_{C_1^n}+4n\mu_{K_1}, \quad \mu_{C_2}=-\mu_{C_2^n}+6\mu_{K_1}.
\]
Thus we have the relation between the Alexander polynomials as
\begin{equation}\label{eq:relation-alexander}
\Delta_{K_1\cup C_1^n\cup C_2^n}(x,y,z)=\Delta_{L}(x,x^{4n}y^{-1},x^{6}z^{-1} ).
\end{equation}

Since $\mathrm{lk}(K_1,C_2^n)=\mathrm{lk}(K,C_2)=3$ and $\mathrm{lk}(C_1^n,C_2^n)=\mathrm{lk}(C_1,C_2)=0$,
 the Torres condition \cite{T} gives
\begin{align*}
\Delta_{K_1\cup C_1^n\cup C_2^n}(x,y,1) &= (x^3y^0-1)\Delta_{K_1\cup C_1^n}(x,y)\\
&= (x^3-1)\Delta_{K_1\cup C_1^n}(x,y).
\end{align*}

Furthermore, since $\mathrm{lk}(K_1,C_1^n)=\mathrm{lk}(K,C_1)=4$, 
\[
\Delta_{K_1\cup C_1^n}(x,1)=\frac{x^4-1}{x-1}\Delta_{K_1}(x).
\]
Thus
\[
\Delta_{K_1}(x)=\frac{x-1}{x^4-1}\Delta_{K_1\cup C_1^n}(x,1)=\frac{x-1}{(x^4-1)(x^3-1)}\Delta_{K_1\cup C_1^n\cup C_2^n}(x,1,1).
\]
Then the relation (\ref{eq:relation-alexander}) gives
\[
\begin{split}
\Delta_{K_1}(t)&=\frac{t-1}{(t^4-1)(t^3-1)} \Delta_L(t,t^{4n},t^6)\\
&=\frac{t-1}{(t^4-1)(t^3-1)}( t^{12n+18}-t^{12n+15}+t^{8n+15}-t^{8n+14}
+t^{8n+11}-t^{8n+9}\\
&\quad +t^{4n+9}-t^{4n+7}+t^{4n+4}-t^{4n+3}+t^3-1)\\
&=\frac{1}{t^3+t^2+t+1}\cdot \frac{1}{t^3-1}
(
t^{12n+15}(t^3-1)+t^{8n+9}(t^6-1)-t^{8n+11}(t^3-1)\\
&\quad +t^{4n+3}(t^6-1)-t^{4n+4}(t^3-1)+(t^3-1)
)\\
&=\frac{1}{t^3+t^2+t+1}
( t^{12n+15}-t^{8n+11}-t^{4n+4}+1
+(t^{8n+9}+t^{4n+3})(t^3+1) ).
\end{split}
\]

We put
\begin{align*}
A_1&= \sum_{i=0}^n (t^{8n+12+4i}-t^{8n+11+4i}),\\
A_2&=t^{8n+9}-t^{8n+8},\\
A_3&=\sum_{i=0}^n (t^{4n+6+4i}-t^{4n+4+4i}),\\
A_4&=t^{4n+3}-t^{4n+1},\\
A_5&=\sum_{i=0}^{n-1}(t^{4+4i}-t^{1+4i}).
\end{align*}

Then a direct calculation shows 
\begin{align*}
(t^3+t^2+t+1)A_1& =t^{12n+15}-t^{8n+11}, \\
(t^3+t^2+t+1)A_2& =t^{8n+12}-t^{8n+8},\\
(t^3+t^2+t+1)A_3&=(t^{8n+9}-t^{4n+5})+(t^{8n+8}-t^{4n+4}),\\
(t^3+t^2+t+1)A_4& =(t^{4n+6}-t^{4n+2})+(t^{4n+5}-t^{4n+1}),\\
(t^3+t^2+t+1)A_5&=(t^{4n+3}-t^3)+(t^{4n+2}-t^2)+(t^{4n+1}-t).
\end{align*}
Thus
\[
\begin{split}
 (t^3+t^2+t+1)(A_1+A_2+A_3+A_4&+A_5+1)=
 t^{12n+15}-t^{8n+11}-t^{4n+4}+1\\
 &+ t^{8n+12}+t^{8n+9}+t^{4n+6}+t^{4n+3}.
\end{split}
\]
We have the conclusion $\Delta_{K_1}(t)=A_1+A_2+A_3+A_4+A_5+1$ as desired.
\end{proof}

\begin{theorem}\label{thm:alex2}
The Alexander polynomial of $K_2$ is  given as
\[
\begin{split}
\Delta_{K_2}(t) &=
\sum_{i=0}^n (t^{8n+12 +4i}-t^{8n+11+4i})+(t^{8n+9}-t^{8n+8})+\sum_{i=0}^{2n-1}(t^{4n+8+2i}-t^{4n+7+2i})\\
&\quad +(t^{4n+6}-t^{4n+4})+(t^{4n+3}-t^{4n+1})+
\sum_{i=0}^{n-1}(t^{4+4i}-t^{1+4i})+1.
\end{split}
\]
\end{theorem}

\begin{proof}
The argument is very similar to the proof of Theorem \ref{thm:alex1}, so we omit the details.

Let $L=K\cup C_1\cup C_2$ be the oriented link illustrated in Figure \ref{fig:knots}(2).
Its multivariable Alexander polynomial is 
\[
\begin{split}
\Delta_L(x,y,z) &=x^6y^3z^2-x^3y^3z^2+x^4y^2z+x^5yz+x^3y^2z^2-x^3y^2z\\
&-x^4yz-x^2y^2z^2+x^4y+x^2y^2z+x^3yz-x^3y-xy^2z-x^2yz+x^3-1.
\end{split}
\]
where $x,y,z$ correspond to the meridians of $K,C_1,C_2$, respectively.

Then
\[
\begin{split}
\Delta_{K_2}(t)&=\frac{t-1}{(t^4-1)(t^3-1)}\Delta_L(t,t^{4n},t^6)\\
&= \frac{t-1}{(t^4-1)(t^3-1)}
( 
t^{12n+18}-t^{12n+15}+t^{8n+10}+t^{4n+11}+t^{8n+15}\\
&\quad -t^{8n+9}-t^{4n+10}-t^{8n+14}+t^{4n+4}+t^{8n+8}+t^{4n+9}-t^{4n+3}\\
&\quad -t^{8n+7}-t^{4n+8}+t^3-1
)\\
&=
\frac{1}{t^3+t^2+t+1}\cdot \frac{1}{t^3-1}
(
(t^{12n+15}+t^{8n+7}+t^{4n+8}+1)(t^3-1)\\
&\quad +(t^{8n+9}-t^{4n+4}-t^{8n+8}+t^{4n+3})(t^6-1)
)\\
&=
\frac{1}{t^3+t^2+t+1}(t^{12n+15}+t^{8n+7}+t^{4n+8}+1\\
&\qquad +(t^{8n+9}-t^{4n+4}-t^{8n+8}+t^{4n+3})(t^3+1)).
\end{split}
\]
Again,
we put
\begin{align*}
B_1&=\sum_{i=0}^n (t^{8n+12 +4i}-t^{8n+11+4i}),\\
B_2&=t^{8n+9}-t^{8n+8},\\
B_3&=\sum_{i=0}^{2n-1}(t^{4n+8+2i}-t^{4n+7+2i}),\\
B_4&=t^{4n+6}-t^{4n+4},\\
B_5&=t^{4n+3}-t^{4n+1},\\
B_6&=\sum_{i=0}^{n-1}(t^{4+4i}-t^{1+4i}).
\end{align*}
A direct calculation shows
\begin{align*}
(t^3+t^2+t+1)B_1&=t^{12n+15}-t^{8n+11},\\
(t^3+t^2+t+1)B_2&=t^{8n+12}-t^{8n+8},\\
(t^3+t^2+t+1)B_3&= (t^{8n+9}-t^{4n+9})+(t^{8n+7}-t^{4n+7}),\\
(t^3+t^2+t+1)B_4&=(t^{4n+9}-t^{4n+5})+(t^{4n+8}-t^{4n+4}),\\
(t^3+t^2+t+1)B_5&=(t^{4n+6}-t^{4n+2})+(t^{4n+5}-t^{4n+1}),\\
(t^3+t^2+t+1)B_6&= (t^{4n+3}-t^3)+(t^{4n+2}-t^2)+(t^{4n+1}-t).  
\end{align*}
This shows that 
\[
\begin{split}
(t^3+t^2+t+1)&(B_1+B_2+B_3+B_4+B_5+B_6+1)
=t^{12n+15}+t^{8n+7}+t^{4n+8}+1\\
&\quad +(t^{8n+9}-t^{4n+4}-t^{8n+8}+t^{4n+3})(t^3+1).
\end{split}
\]
Thus $\Delta_{K_2}(t)=B_1+B_2+B_3+B_4+B_5+B_6+1$ as desired.
\end{proof}

We recall the notion of formal semigroup for an L--space knot \cite{W}.
Let $K$ be an L--space knot in the $3$--sphere.
Then the Alexander polynomial of $K$ has a form of
\begin{equation}\label{eq:alex}
\Delta_K(t)=1-t^{a_1}+t^{a_2}+\dots -t^{a_{k-1}}+t^{a_k},
\end{equation}
where $1=a_1<a_2<\dots <a_k=2g(K)$, and  $g(K)$ is the genus of $K$ \cite{OS}.
We expand the Alexander function into a formal power series as
\begin{equation}\label{eq:ps}
\frac{\Delta_K(t)}{1-t}=\sum_{s\in \mathcal{S}_K} t^s.
\end{equation}
(This is called the Milnor torsion in \cite{DR}.)
The set $\mathcal{S}_K$ is a subset of non-negative integers, called
the \textit{formal semigroup\/} of $K$.
For example, for a torus knot $T(p,q)\ (1<p<q)$,
its formal semigroup is known to be the actual semigroup of rank two,
\[
\langle p, q\rangle  =\{ap+bq \mid a, b\ge 0\}
\]
(see \cite{BL0,W}).
If an L--space knot is an iterated torus knot, then
its formal semigroup is also a semigroup \cite{W}, but
in general, the formal semigroup of a hyperbolic L--space knot
is hardly a semigroup \cite{BK,Te}.

Let $\mathbb{Z}_{\ge m}=\{ i\in \mathbb{Z} \mid i\ge m\}$ and
$\mathbb{Z}_{<0}=\{i\in \mathbb{Z} \mid i<0\}$.

\begin{proposition}\label{prop:formal1}
The formal semigroup of $K_1$ is given as
\[
\begin{split}
\mathcal{S}_{K_1}&=\{0,4,8,\dots,4n\} \cup \{ 4n+3\}\\
&\quad  \cup
\{ 4n+6,4n+7,4n+10,4n+11,\dots, 8n+6,8n+7\}\cup \{8n+9,8n+10 \}\\
&\quad 
\cup \{8n+12,8n+13,8n+14, 8n+16,8n+17,8n+18,\\
&\qquad  \dots, 12n+8,12n+9,12n+10 \}\cup \mathbb{Z}_{\ge 12n+12}.
\end{split}
\]
\end{proposition}

\begin{proof}
We use $A_1,A_2,\dots, A_5$ in the proof of Theorem \ref{thm:alex1}.
For 
\[
\frac{\Delta_{K_1}}{1-t}=\frac{A_1}{1-t}+\frac{A_2}{1-t}+\frac{A_3}{1-t}+\frac{A_4}{1-t}+\frac{A_5}{1-t}+
\frac{1}{1-t},
\]
we expand each term as follows;
\begin{align*}
\frac{A_1}{1-t}&=-\sum_{i=0}^nt^{8n+11+4i},\\
\frac{A_2}{1-t}&=-t^{8n+8},\\
\frac{A_3}{1-t}&=-\sum_{i=0}^n(t^{4n+5+4i}+t^{4n+4+4i}),\\
\frac{A_4}{1-t}&=-t^{4n+2}-t^{4n+1},\\
\frac{A_5}{1-t}&=-\sum_{i=0}^{n-1}(t^{3+4i}+t^{2+4i}+t^{1+4i}),\\
\frac{1}{1-t}&=1+t+t^2+t^3+\dots.
\end{align*}
The conclusion immediately follows from these.
\end{proof}

\begin{proposition}\label{prop:formal2}
The formal semigroup of $K_2$ is given as
\[
\begin{split}
\mathcal{S}_{K_2}&=\{0,4,8,\dots,4n\} \cup \{ 4n+3\}\\
&\quad  \cup
\{ 4n+6,4n+8,4n+10,\dots, 8n+4\}\cup \{8n+6,8n+7,8n+9,8n+10 \}\\
&\quad 
\cup \{8n+12,8n+13,8n+14, 8n+16,8n+17,8n+18,\\
&\qquad  \dots, 12n+8,12n+9,12n+10 \}\cup \mathbb{Z}_{\ge 12n+12}.
\end{split}
\]
\end{proposition}

\begin{proof}
The argument is similar to the proof of Proposition \ref{prop:formal1}.
For $B_1,B_2,\dots,B_6$ in the proof of Theorem \ref{thm:alex2},
we expand
\begin{align*}
\frac{B_1}{1-t}&= -\sum_{i=0}^n t^{8n+11+4i},\\
\frac{B_2}{1-t}&=-t^{8n+8},\\
\frac{B_3}{1-t}&=-\sum_{i=0}^{2n-1}t^{4n+7+2i},\\
\frac{B_4}{1-t}&=-t^{4n+4}-t^{4n+5},\\
\frac{B_5}{1-t}&=-t^{4n+1}-t^{4n+2},\\
\frac{B_6}{1-t}&=-\sum_{i=0}^{n-1}(t^{3+4i}+t^{2+4i}+t^{1+4i}).
\end{align*}
Then the conclusion follows from these again.
\end{proof}

\begin{corollary}\label{cor:formal}
For $i=1,2$, the formal semigroup of $K_i$ is not a semigroup.
\end{corollary}

\begin{proof}
By Propositions \ref{prop:formal1} and \ref{prop:formal2},
we see that 
$4\in  \mathcal{S}_{K_i}$ but $4n+4\not \in\mathcal{S}_{K_i}$.
Hence $\mathcal{S}_{K_i}$ is not closed under the addition, so is not a semigroup.
\end{proof}

\begin{lemma}\label{lem:hyp}
Both of $K_1$ and $K_2$ are hyperbolic.
\end{lemma}

\begin{proof}
By Corollary \ref{cor:formal},
the formal semigroup of $K_i$ is not a semigroup.
Hence $K_i$ is not a torus knot, because the formal semigroup of a torus knot is a semigroup.

Assume for a contradiction that $K_i$ is a satellite knot.
Since $K_i$ is the closure of a $4$--braid, its bridge number is at most four.
By \cite{Sc}, it is equal to four.
Moreover, the companion is a $2$--bridge knot and the pattern knot has wrapping number two.
We know that both of the companion and the pattern knot are L--space knots and the pattern is braided by \cite{BM,Ho}.
Thus the companion is a $2$--bridge torus knot \cite{OS}, and $K_i$ is its $2$--cable.
By \cite{W},
the formal semigroup of an iterated torus L--space knot is a semigroup, which contradicts Corollary \ref{cor:formal}.
We have thus shown that $K_i$ is hyperbolic.
\end{proof}

\section{Upsilon invariants}\label{sec:upsilon}

In this section, we verify that the Upsilon invariants of $K_1$ and $K_2$
are the same.
We will not calculate the Upsilon invariants.
Instead, we determine the gap functions defined later.
For an L--space knot, the Upsilon invariant is the Legendre--Fenchel transform of the gap function \cite{BH}.
Hence if the gap functions of $K_1$ and $K_2$ share the same convex hull, then
their Upsilon invariants also coincide.

First, we quickly review the Legendre--Fenchel transformation.

For a function $f\colon \mathbb{R}\to \mathbb{R}$, the Legendre--Fenchel transform $f^*\colon \mathbb{R}\to \mathbb{R}\cup \{\infty\}$
is defined as
\[
f^*(t)=\sup_{x\in\mathbb{R}}\{ tx - f(x) \}.
\]
The domain of $f^*$ is the set $\{t \mid f^*(t)<\infty\}$.

The Legendre transform is defined only for differentiable  convex functions, but
the Legendre--Fenchel transform can be defined even for non-convex functions with non-differentiable points.
The transform $f^*$ is always a convex function.
Hence, if $f$ is not convex, then the double Legendre--Fenchel transform $f^{**}$ does not return $f$.
In this case, $f^{**}$ gives the convex hull of the function $f$.
Thus we see that $f^*$ depends only on the convex hull of $f$.

Next, we recall the notion of gap function introduced in \cite{BL0}.

Let $K$ be an L--space knot with formal semigroup $\mathcal{S}_K$.
Then $\mathcal{G}_{K}=\mathbb{Z}-\mathcal{S}_{K}$ is called the \textit{gap set}.
In fact, $\mathcal{G}_K=\mathbb{Z}_{<0}\cup \{a_1,a_2,\dots,a_g\}$,
where $g=g(K)$, and $0<a_1<a_2<\dots < a_g$.
The part $a_1,a_2,\dots, a_g$ is called the gap sequence.
Then it is easy to restore the Alexander polynomial as
\[
\Delta_K(t)=1+(t-1)(t^{a_1}+t^{a_2}+\dots +t^{a_g}).
\]
From the gap set $\mathcal{G}_K$,
we define the function $I\colon \mathbb{Z}\to \mathbb{Z}_{\ge 0}$ by
\[
I(m)=\# \{ i\in \mathcal{G}_K \mid i\ge m\},
\]
and let $J(m)=I(m+g)$.
Then we extend $J(m)$ linearly to obtain a piecewise linear function on $\mathbb{R}$.
That is, for $k\in \mathbb{Z}$,  if $J(k)=J(k+1)$, then $J(x)=J(k)$ on $[k,k+1]$, and
if $J(k+1)=J(k)-1$, then $J(k+x)=J(k)-x$ for $0\le x\le 1$.
Borodzik and Hedden \cite{BH} showed that 
the Upsilon invariant of $K$ is the Legendre--Fenchel transform of the function $2J(-m)$.
We call this function $2J(-m)$ the \textit{gap function\/} of $K$.

\begin{example}\label{ex:237}
Let $K$ be the $(-2,3,7)$--pretzel knot.
It admits a lens space surgery, so is an L--space knot.
Also, it has genus $5$.
The Alexander polynomial $\Delta_K(t)$ is $1-t+t^3-t^4+t^5-t^6+t^7-t^9+t^{10}$.
Then $\mathcal{S}_K=\{0,3,5,7,8\}\cup \mathbb{Z}_{\ge 10}$, and
$\mathcal{G}_K=\mathbb{Z}_{<0}\cup \{1,2,4,6,9\}$.
Tables \ref{table:I(m)}
and \ref{table:gap}  show the values of $I(m)$ and the gap function $2J(-m)$.

\begin{table}[h]
\begin{tabular}{c|cccccccccccccc}
$m$ & $\ge 10$ & 9 & 8 & 7 & 6 & 5 & 4 & 3 & 2 & 1 & 0 & $-1$  & $-2$ & $\dots$ \\
\hline
$I(m)$ & 0 & 1 & 1 & 1 & 2 & 2 & 3 & 3 & 4 & 5 & 5 & 6 & 7 & $\dots$ 
\end{tabular}
\caption{$I(m)$ for the $(-2,3,7)$--pretzel knot.}\label{table:I(m)}
\end{table}

\begin{table}[h]
\begin{tabular}{c|cccccccccccccc}
$m$ & $\le -5$ & $-4$ & $-3$ & $-2$ & $-1$ & 0 &1 & 2 & 3 & 4 & 5 & 6  & 7 & $\dots$ \\
\hline
$2J(-m)$ & 0 & 2 & 2 & 2 & 4 & 4 & 6 & 6 & 8 & 10 & 10 & 12 & 14 & $\dots$ 
\end{tabular}
\caption{The gap function $2J(-m)$ for the $(-2,3,7)$--pretzel knot.}\label{table:gap}
\end{table}

Figure \ref{fig:grid} shows the graph of the gap function $2J(-m)$
and its convex hull (broken line).
Here, the convex hull $f(x)$ of the gap function is given by
\[
f(x)=
\begin{cases}
0 & \text{for $x\le -5$},\\
\frac{2}{3}(x+5) & \text{for $-5\le x\le -2$}, \\
x+4 & \text{for $-2\le x \le 2$}, \\
\frac{4}{3}(x-5)+10 & \text{for $2\le x\le 5$}, \\
2x & \text{for $5\le x$}.
\end{cases}
\]

\begin{figure}[htpb]
\begin{center}
\includegraphics[scale=1.5]{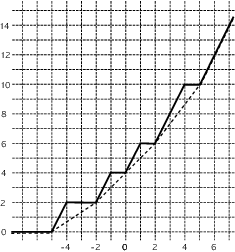}
\caption{The graph of the gap function $2J(-m)$ for the $(-2,3,7)$--pretzel knot and
its convex hull (broken line).}\label{fig:grid}
\end{center}
\end{figure}

Then the Legendre--Fenchel transformation gives the Upsilon invariant
\[
\Upsilon_K(t)=
\begin{cases}
-5t & \text{for $0\le t \le \frac{2}{3}$}, \\
-2t-2 & \text{for $\frac{2}{3}\le t \le 1$}, \\
2t-6 & \text{for $1\le t \le \frac{4}{3}$}, \\
5t-10 & \text{for $\frac{4}{3}\le t\le 2$}.
\end{cases}
\]
\end{example}

In general, the gap function of an L--space knot has a specific property.
\begin{itemize}
\item
The slope of each segment of the graph is $0$ or $2$.
\end{itemize}
Although this observation  is easy to see, we will use it essentially in Section \ref{sec:restore}
with further investigation.

Now, we calculate the gap functions of $K_1$ and $K_2$.

From Proposition \ref{prop:formal1},
the gap set $\mathcal{G}_{K_1}$ is 
\[
\begin{split}
\mathbb{Z}_{<0}&\cup \{1,2,3,5,6,7,\dots,4n-3,4n-2,4n-1\}\cup \{4n+1,4n+2\}\\
& \cup \{4n+4,4n+5,4n+8,4n+9,4n+12,4n+13,\dots, 8n+4,8n+5\}\\
& \cup \{8n+8\}
\cup \{8n+11,8n+15,8n+19,\dots,12n+7,12n+11\}.
\end{split}
\]

Hence the values of $I(m)$ is given as in Table \ref{table:Im1}.
When $m$ is an integer not in the table, $I(m)$ takes the same value
as the nearest $m'$ with $m'>m$.
For example, $I(m)=I(12n+11)=1$ for $m=12n+10,12n+9,12n+8$.

{\scriptsize
\begin{table}[h]
\begin{tabular}{c|c|cccc|c}
$m$    & $\ge 12n+12$ & $12n+11$ & $12n+7$ & $\dots$ & $8n+11$ & $8n+8$ \\
\hline
$I(m)$ & 0  &1  &  2 &  $\dots$ &  $n+1$ &  $n+2$ 
\end{tabular}
\vskip 10pt
\begin{tabular}{ccccccc|ccccc}
 $8n+5$ & $8n+4$ & $8n+1$ & $8n$ & $\dots$ & $4n+5$ & $4n+4$ & $4n+2$ & $4n+1$ 
\\
\hline
  $n+3$ & $n+4$ & $n+5$ & $n+6$ & $\dots$ & $3n+3$ & $3n+4$  & $3n+5$ & $3n+6$ 
\end{tabular}
\vskip 10pt
\begin{tabular}{ccccccc|ccc}
 $4n-1$ & $4n-2$ & $4n-3$ & $\dots$ & 3 & 2 & 1 & $-1$ & $-2$ & $\dots$ \\
\hline
 $3n+7$ & $3n+8$ & $3n+9$ & $\dots$ & $6n+4$ & $6n+5$ & $ 6n+6$  & $6n+7$ & $6n+8$ & $\dots$
\end{tabular} 
\caption{The function $I(m)$ for $K_1$.}\label{table:Im1}
\end{table}
}

Let $J(m)=I(m+g)$ with $g=6n+6$.
Then the gap function $2J(-m)$ takes the values as in Table \ref{table:Im1gap}.

{\scriptsize
\begin{table}[h]
\begin{tabular}{c|c|cccc|c}
$m$    & $\le -6n-6$ & $-6n-5$ & $-6n-1$ & $\dots$ & $-2n-5$ & $-2n-2$ \\
\hline
$2J(-m)$ & 0  &2  & 4 &  $\dots$ &  $2n+2$ &  $2n+4$ 
\end{tabular}
\vskip 10pt
\begin{tabular}{ccccccc|ccccc}
 $-2n+1$ & $-2n+2$ & $-2n+5$ & $-2n+6$ & $\dots$ & $2n+1$ & $2n+2$ & $2n+4$ & $2n+5$ 
\\
\hline
  $2n+6$ & $2n+8$ & $2n+10$ & $2n+12$ &  $\dots$ & $6n+6$ & $6n+8$  & $6n+10$ & $6n+12$ 
\end{tabular}
\vskip 10pt
\begin{tabular}{ccccccc|ccc}
 $2n+7$ & $2n+8$ & $2n+9$ & $\dots$ & $6n+3$ & $6n+4$ & $6n+5$ & $6n+7$ & $6n+8$ & $\dots$ \\
\hline
 $6n+14$ & $6n+16$ & $6n+18$ & $\dots$ & $12n+8$ & $12n+10$ & $ 12n+12$  & $12n+14$ & $12n+16$ & $\dots$
\end{tabular} 
\caption{The gap function $2J(-m)$ for $K_1$.}\label{table:Im1gap}
\end{table}
}

Figure \ref{fig:K1gap} shows the graph of the gap function $2J(-m)$ of $K_1$ when $n=1$.

\begin{figure}[htpb]
\begin{center}
\includegraphics[scale=1.3]{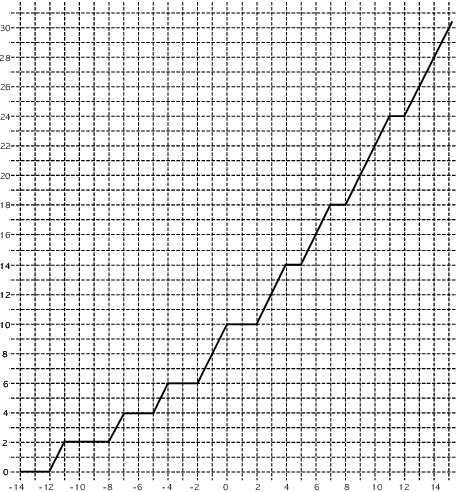}
\caption{The graph of the gap function of $K_1$ with $n=1$.}\label{fig:K1gap}
\end{center}
\end{figure}

Similarly, the gap set $\mathcal{G}_{K_2}$ is
\[
\begin{split}
\mathbb{Z}_{<0}&\cup \{1,2,3,5,6,7,\dots,4n-3,4n-2,4n-1\}\cup \{4n+1,4n+2\}\\
& \cup \{4n+4,4n+5\}
\cup \{4n+7,4n+9,\dots, 8n+3,8n+5\}\\
& \cup \{8n+8,8n+11\}\cup \{8n+15,8n+19,\dots,12n+7,12n+11\}
\end{split}
\]
from Proposition \ref{prop:formal2}.

The values of $I(m)$ and the gap function $2J(-m)$ are given as in 
Tables \ref{table:Im2} and \ref{table:Im2gap}.

{\scriptsize
\begin{table}[h]
\begin{tabular}{c|c|cccc|cc}
$m$    & $\ge 12n+12$ & $12n+11$ & $12n+7$ & $\dots$ & $8n+15$  & $8n+11$ & $8n+8$ \\
\hline
$I(m)$ & 0  &1  &  2 &  $\dots$ &  $n$ &  $n+1$ &  $n+2$ 
\end{tabular}
\vskip 10pt
\begin{tabular}{ccccc|cc|ccccc}
 $8n+5$ & $8n+3$ & $8n+1$ & $\dots$ & $4n+7$ & $4n+5$ & $4n+4$ & $4n+2$ & $4n+1$ 
\\
\hline
  $n+3$ & $n+4$ & $n+5$ & $\dots$ & $3n+2$ & $3n+3$ & $3n+4$  & $3n+5$ & $3n+6$ 
\end{tabular}
\vskip 10pt
\begin{tabular}{ccccccc|ccc}
 $4n-1$ & $4n-2$ & $4n-3$ & $\dots$ & 3 & 2 & 1 & $-1$ & $-2$ & $\dots$ \\
\hline
 $3n+7$ & $3n+8$ & $3n+9$ & $\dots$ & $6n+4$ & $6n+5$ & $ 6n+6$  & $6n+7$ & $6n+8$ & $\dots$
\end{tabular} 
\caption{The function $I(m)$ for $K_2$.}\label{table:Im2}
\end{table}

\begin{table}[h]
\begin{tabular}{c|c|cccc|cc}
$m$    & $\le -6n-6$ & $-6n-5$ & $-6n-1$ & $\dots$ & $-2n-9$ & $-2n-5$ & $-2n-2$ \\
\hline
$2J(-m)$ & 0  &2  & 4 &  $\dots$ &  $2n$ &  $2n+2$  & $2n+4$ 
\end{tabular}
\vskip 10pt
\begin{tabular}{ccccc|cc|ccccc}
 $-2n+1$ & $-2n+3$ & $-2n+5$ & $\dots$ & $2n-1$ & $2n+1$ & $2n+2$ & $2n+4$ & $2n+5$ 
\\
\hline
  $2n+6$ & $2n+8$ & $2n+10$ & $\dots$ &  $6n+4$ & $6n+6$ & $6n+8$  & $6n+10$ & $6n+12$ 
\end{tabular}
\vskip 10pt
\begin{tabular}{ccccccc|ccc}
 $2n+7$ & $2n+8$ & $2n+9$ & $\dots$ & $6n+3$ & $6n+4$ & $6n+5$ & $6n+7$ & $6n+8$ & $\dots$ \\
\hline
 $6n+14$ & $6n+16$ & $6n+18$ & $\dots$ & $12n+8$ & $12n+10$ & $ 12n+12$  & $12n+14$ & $12n+16$ & $\dots$
\end{tabular} 
\caption{The gap function $2J(-m)$ for $K_2$.}\label{table:Im2gap}
\end{table}
}


\begin{figure}[htpb]
\begin{center}
\includegraphics[scale=1]{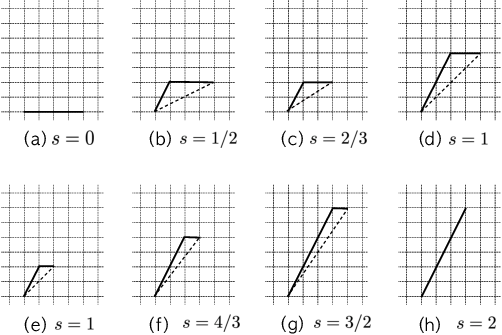}
\caption{The parts of the graph of  a gap function.  The broken lines show
the parts of convex hull with slope $s$.}\label{fig:part}
\end{center}
\end{figure}

\begin{lemma}\label{lem:convex}
For $i=1,2$, 
the convex hull $f(x)$ of the gap function $2J(-m)$ for $K_i$ is given by
\[
f(x)=
\begin{cases}
0 & \text{for $x\le -6n-6$}, \\
\frac{1}{2}(x+6n+6) & \text{for $-6n-6\le x \le -2n-6$},\\
\frac{2}{3}(x+2n+6)+2n   &  \text{for $-2n-6\le x \le -2n$}, \\
x+4n+4 & \text{for $-2n \le x \le 2n$},\\
\frac{4}{3}(x-2n)+6n+4 & \text{for $2n\le x \le 2n+6$}, \\
\frac{3}{2}(x-2n-6)+6n+12  & \text{for $2n+6\le x \le 6n+6$},\\
2x & \text{for $6n+6\le x$}.
\end{cases}
\]
\end{lemma}

\begin{proof}
Consider  the gap function of $K_1$.
Let $f$ be the convex hull.
From Table \ref{table:Im1gap},
it is obvious that $f(x)=0$ for $x\le -6n-6$ and
$f(x)=2x$ for $x\ge 6n+6$.

On the interval $[-6n-6,-6n-2]$, the gap function has the branch as shown in Figure \ref{fig:part}(b).
It repeats on the intervals $[-6n-2,-6n+2], \dots, [-2n-10,-2n-6]$.
Thus $f(x)=\frac{1}{2} (x+6n+6)$ on $[-6n-6,-2n-6]$.

On $[-2n-6,-2n-3]$ and $[-2n-3,-2n]$, the branch is of Figure \ref{fig:part}(c).
Hence $f(x)=\frac{2}{3}(x+2n+6)+2n$ on $[-2n-6,-2n]$.

Similarly, the branch of Figure \ref{fig:part}(d) repeats on the intervals 
$[-2n,2n+4], [-2n+4,-2n+6], \dots, [2n-4,2n]$.
This gives $f(x)=x+4n+4$ on $[-2n,2n]$.

On $[2n,2n+3]$ and $[2n+3,2n+6]$, the branch of Figure \ref{fig:part}(f) appears.
Thus $f(x)=\frac{4}{3}(x-2n)+6n+4$ on $[2n,2n+6]$.

Finally, the branch of Figure \ref{fig:part}(g) repeats on $[2n+6,2n+10], \dots, [6n+2,6n+6]$.
Then $f(x)=\frac{3}{2}(x-2n-6)+6n+12$ on $[2n+6,6n+6]$.
We have thus shown that the convex hull $f(x)$ is given as claimed for $K_1$.

Next, consider the gap function of $K_2$.
For $x\le -2n$, the situation is the same as $K_1$.

On $[-2n,-2n+2]$, the branch of Figure \ref{fig:part}(e) appears.
This branch repeats on $[-2n+2,-2n+4],\dots, [2n-2,2n]$.
However, the convex hull is the same as $K_1$.

For the remaining range $x\ge 2n$,  the gap function is the same as one of $K_1$.
In conclusion, the gap functions of $K_1$ and $K_2$ are distinct only on $[-2n,2n]$, but
their convex hulls coincide there.
\end{proof}

\begin{corollary}\label{cor:upsilon}
The Upsilon invariants of $K_1$ and $K_2$ coincide.
\end{corollary}

\begin{proof}
The Upsilon invariant is the Legendre--Fenchel transform of the gap function $2J(-m)$.
In fact, it depends only on the convex hull of the gap function.
By Lemma \ref{lem:convex}, $K_1$ and $K_2$ have the same convex hull for their
gap functions.
Thus the conclusion follows.
\end{proof}


\section{The Montesinos trick}\label{sec:mont}

In this section, we verify that $K_1$ and $K_2$ admit positive Dehn surgeries yielding
L--spaces by using the Montesinos trick \cite{Mon}.
For a surgery diagram on a strongly invertible link,
the Montesinos trick describes the resulting closed $3$--manifold as
the double branched cover of another knot or link obtained
from tangle replacements corresponding to the surgery coefficients on
some link obtained from the quotient of the original strongly invertible link under
the strong involution (see also \cite{MT, Wa}).

In Figure \ref{fig:knots}(1) and (2),
each link $K\cup C_1\cup C_2$ is placed in a strongly invertible position,
where the dotted line indicates the axis of the involution.

\begin{lemma}\label{lem:K1mont}
For $K_1$, $(16n+21)$--surgery yields an L--space. 
\end{lemma}

\begin{proof}
Assign the surgery coefficient $3$ on $K$ in Figure \ref{fig:knots}(1).
After performing $(-1/n)$--surgery on $C_1$ and $(-1/2)$--surgery on $C_2$,
our knot $K_1$ has surgery coefficient $16n+21$.

The left of Figure \ref{fig:mont1-0} shows the knot obtained from the tangle replacements.
In the diagram of Figure \ref{fig:knots}, we should remark that the component $K$ has writhe $3$.
Hence the tangle replacement corresponding to the quotient of $K$ is realized by the $0$--tangle
(depicted as the dotted circle).

Then Figures \ref{fig:mont1-0}, \ref{fig:mont1-1}, and \ref{fig:mont1-2} show
the deformation of the knot.
Finally, we obtain the Montesinos knot $M(-3/7,-1/3,-1/n)$.
Thus the double branched cover is the Seifert fibered manifold $M=M(0;-3/7,-1/3,-1/n)$.
We use the notation of  \cite{LS}.  That is, $M(e_0;r_1,r_2,r_3)$ is obtained by $e_0$--surgery on the unknot
with three meridians having $(-/r_i)$--surgery on the $i$-th one.
Then $-M=M(0;3/7,1/3,1/n)$.
By the criterion of \cite{LM,LS},  $M$ is an L--space.
\end{proof}

\begin{figure}[htpb]
\begin{center}
\includegraphics[scale=0.4]{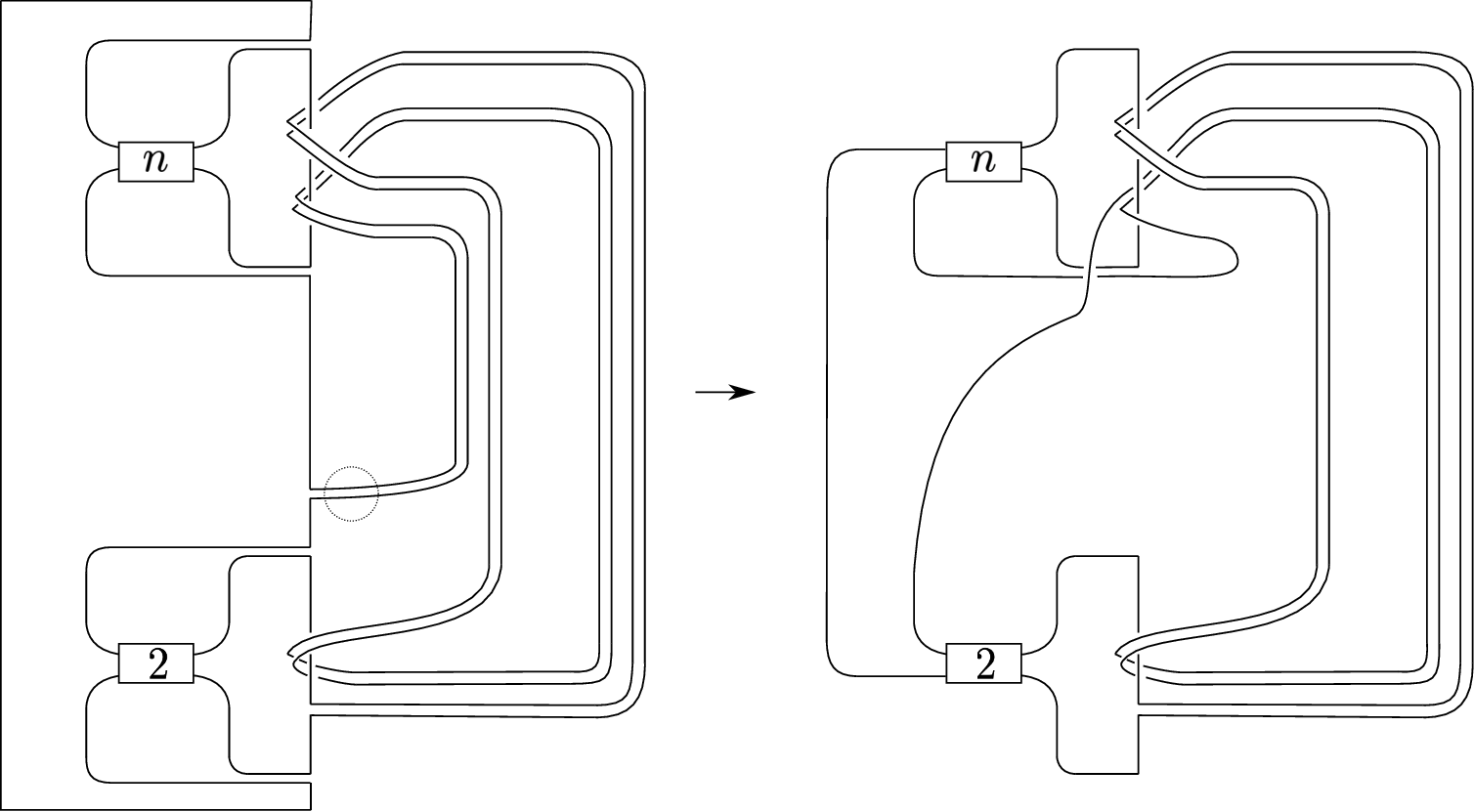}
\caption{The deformation for $K_1$.
Each rectangle box contains  horizontal right-handed half-twists with
indicated number.
}\label{fig:mont1-0}
\end{center}
\end{figure}

\begin{figure}[htpb]
\begin{center}
\includegraphics[scale=0.4]{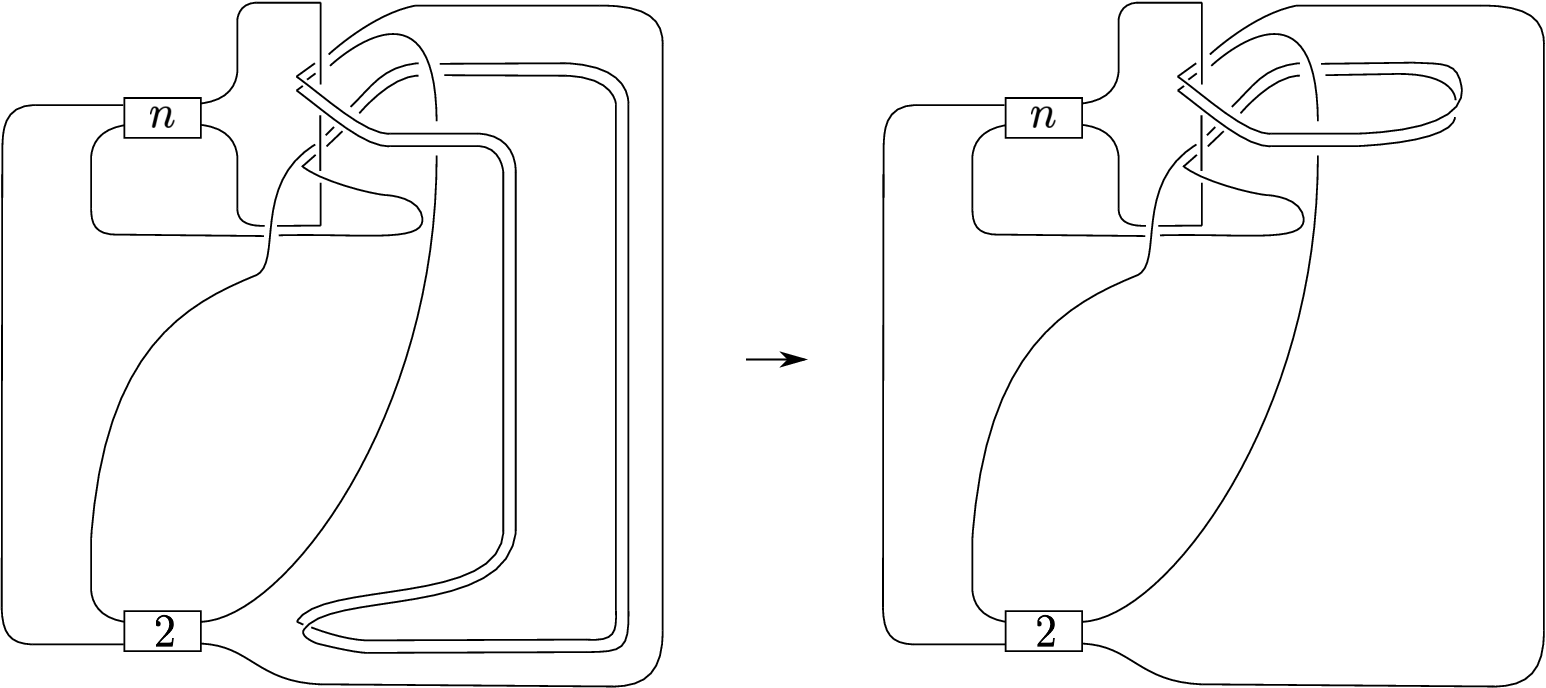}
\caption{The deformation for $K_1$ (continued from Figure \ref{fig:mont1-0}).
}\label{fig:mont1-1}
\end{center}
\end{figure}

\begin{figure}[htpb]
\begin{center}
\includegraphics[scale=0.4]{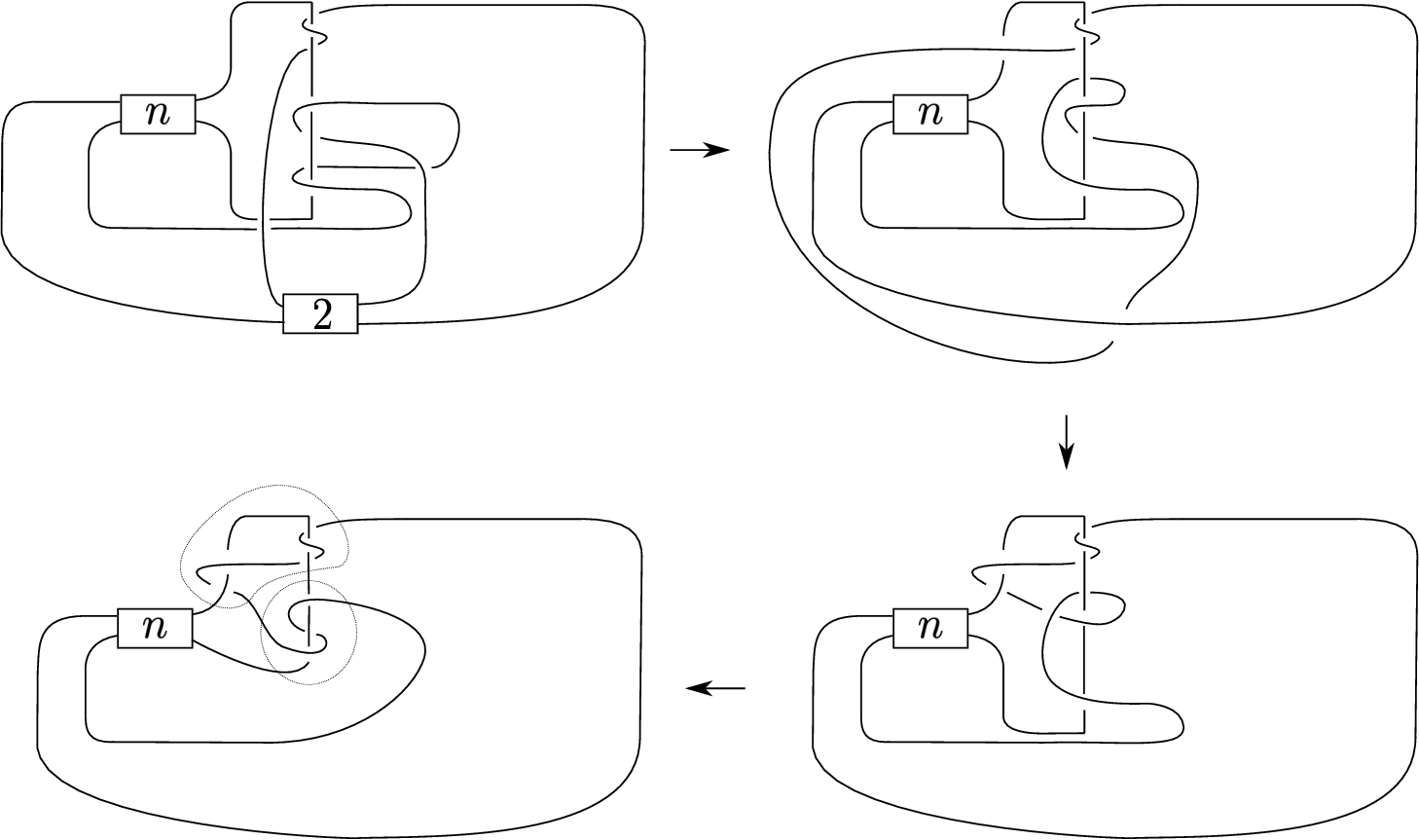}
\caption{The deformation for $K_1$ (continued from Figure \ref{fig:mont1-1}).
The left bottom is the Montesinos knot $M(-3/7, -1/3,-1/n)$.
}\label{fig:mont1-2}
\end{center}
\end{figure}


\begin{lemma}\label{lem:K2mont}
For $K_2$, $(16n+20)$--surgery yields an L--space. 
\end{lemma}

\begin{proof}
Assign the surgery coefficient $2$ on $K$ in Figure \ref{fig:knots}(2).
After performing $(-1/n)$--surgery on $C_1$ and $(-1/2)$--surgery on $C_2$,
$K_2$ has surgery coefficient $16n+20$.

The process is similar to that for $K_1$.
We should remark that the tangle replacement to the quotient of $K$ is realized by
$(-1)$--tangle as depicted in the dotted circle in Figure \ref{fig:mont2-0} (left), because
$K$ has writhe $3$ in the diagram.

\begin{figure}[htpb]
\begin{center}
\includegraphics[scale=0.4]{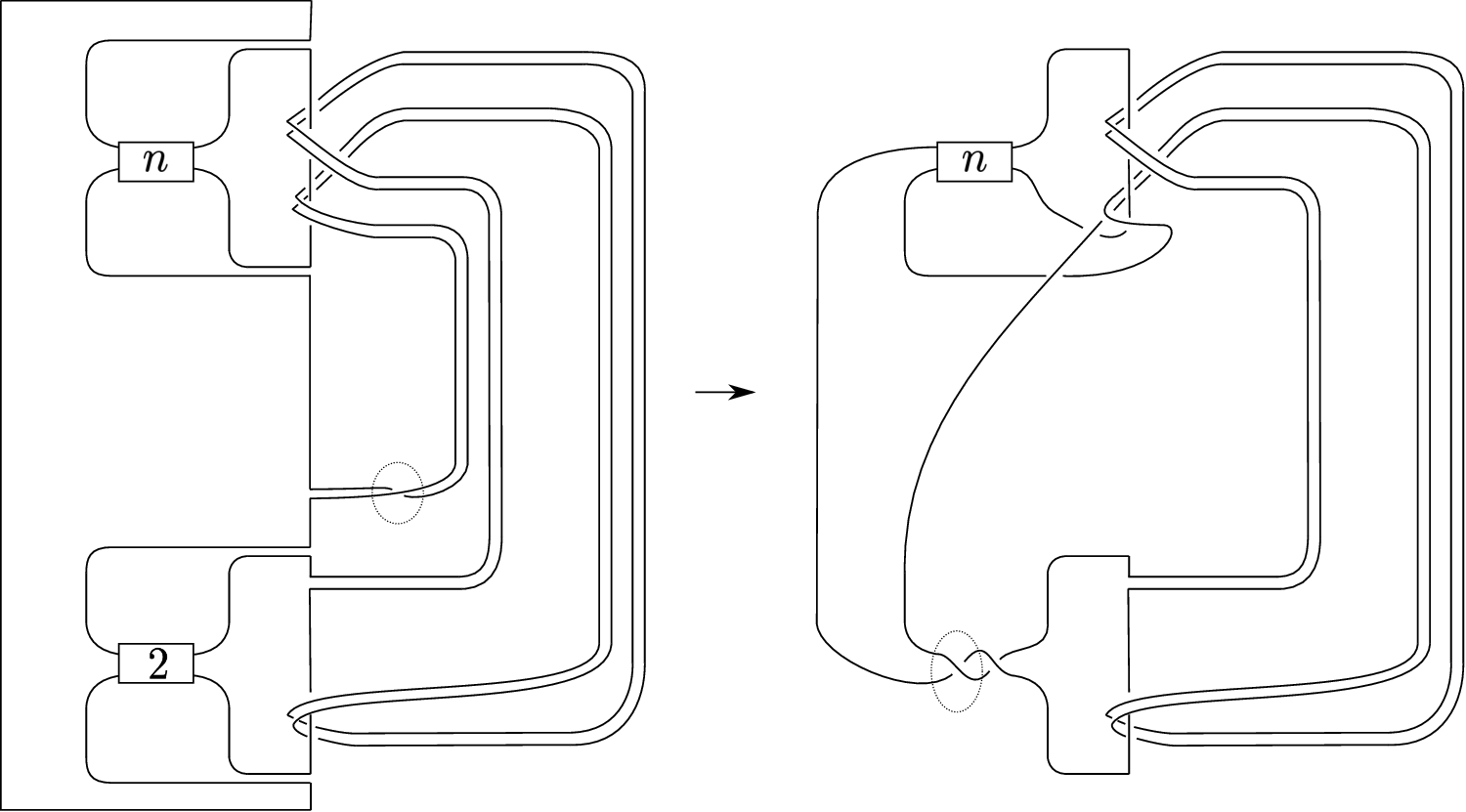}
\caption{The deformation for $K_2$.
Let $\ell$ be the right link.
}\label{fig:mont2-0}
\end{center}
\end{figure}

Let $\ell$ be the link as illustrated in the right of Figure \ref{fig:mont2-0}.
We need to verify that the double branched cover of $\ell$ is an L--space.

For the crossing of $\ell$ encircled in Figure \ref{fig:mont2-0} (right),
we perform two resolutions as shown in Figure \ref{fig:resolution}.
Let $\ell_\infty$ and $\ell_0$ be the resulting knots.
It is straightforward to calculate $\det \ell=16n+20$,
$\det \ell_\infty=9$ and $\det \ell_0=16n+11$ from the checkerboard colorings on the diagrams
of Figures \ref{fig:mont2-0}, \ref{fig:k2mont2-inf2} and \ref{fig:k2mont2-00}.
Thus the equation $\det \ell=\det \ell_\infty+\det \ell_0$ holds.
This implies that if the double branched covers of $\ell_\infty$ and $\ell_0$ are L--spaces, then
so is the double branched cover of $\ell$ (\cite{BGW,OS,OS2}).

\begin{figure}[htpb]
\begin{center}
\includegraphics[scale=0.5]{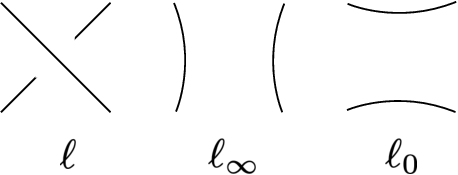}
\caption{Two resolutions.
}\label{fig:resolution}
\end{center}
\end{figure}

\begin{claim}\label{cl:ellinfty}
The knot $\ell_\infty$ is the $(-3,3,n-1)$--pretzel knot.
Its double branched cover is an L--space.
\end{claim}

\begin{proof}
Figures \ref{fig:k2mont2-inf} and \ref{fig:k2mont2-inf2} show that
the knot $\ell_\infty$ is the $(-3,3,n-1)$--pretzel knot.

If $n=1$, then $\ell_\infty$ is the connected sum of torus knots $T(2,3)$ and $T(2,-3)$.
Then the double branched cover is the connected sum of lens spaces $L(3,1)\# L(3,-1)$, which is an L--space.
If $n=2$, then $\ell_\infty$ is the $2$--bridge knot $S(4/9)$, so the double cover is a lens space.
Hence we assume $n>2$.

Since $\ell_\infty$ is the Montesinos knot $M(0;1/3,-1/3,-1/(n-1))$,
its double branched cover $M$ is the Seifert fibered manifold $M(0;1/3,-1/3,-1/(n-1))$.
Then $-M=M(-1; 2/3,1/3,1/(n-1))$.

We use the criterion of \cite{LS}.
If $n>3$, then set $r_1=2/3$, $r_2=1/3$ and $r_3=1/(n-1)$.
Then $1\ge r_1\ge r_2\ge r_3\ge 0$.
If there are no coprime integers $m>a>0$ such that $a/m>r_1$, $(m-a)/m>r_2$ and $1/m>r_3$,
then $-M$ is an L--space.
However, the first two give $2/3<a/m<2/3$, so there are no such integers.

Finally, assume $n=3$.
Set $r_1=2/3$, $r_2=1/2$ and $r_3=1/3$.  Then $r_1\ge r_2\ge r_3$.
If $1/m>r_3=1/3$, then $m<3$.
For $m=2$ and $a=1$, $a/m>r_1=2/3$ does not hold.
Thus there are no coprime integers $m>a>0$ as desired, which implies that $-M$ is an L--space.
\end{proof}

\begin{figure}[htpb]
\begin{center}
\includegraphics[scale=0.4]{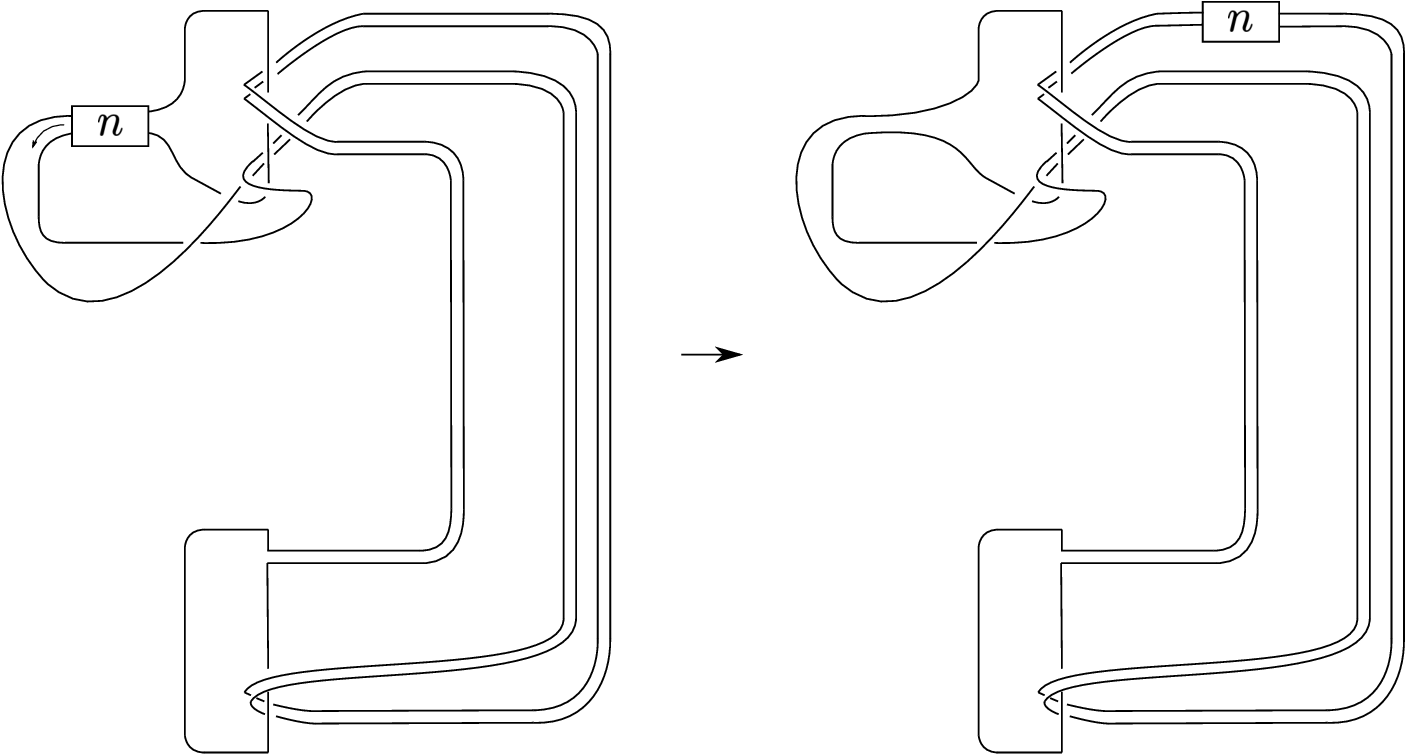}
\caption{A deformation of the knot $\ell_\infty$.
}\label{fig:k2mont2-inf}
\end{center}
\end{figure}

\begin{figure}[htpb]
\begin{center}
\includegraphics[scale=0.4]{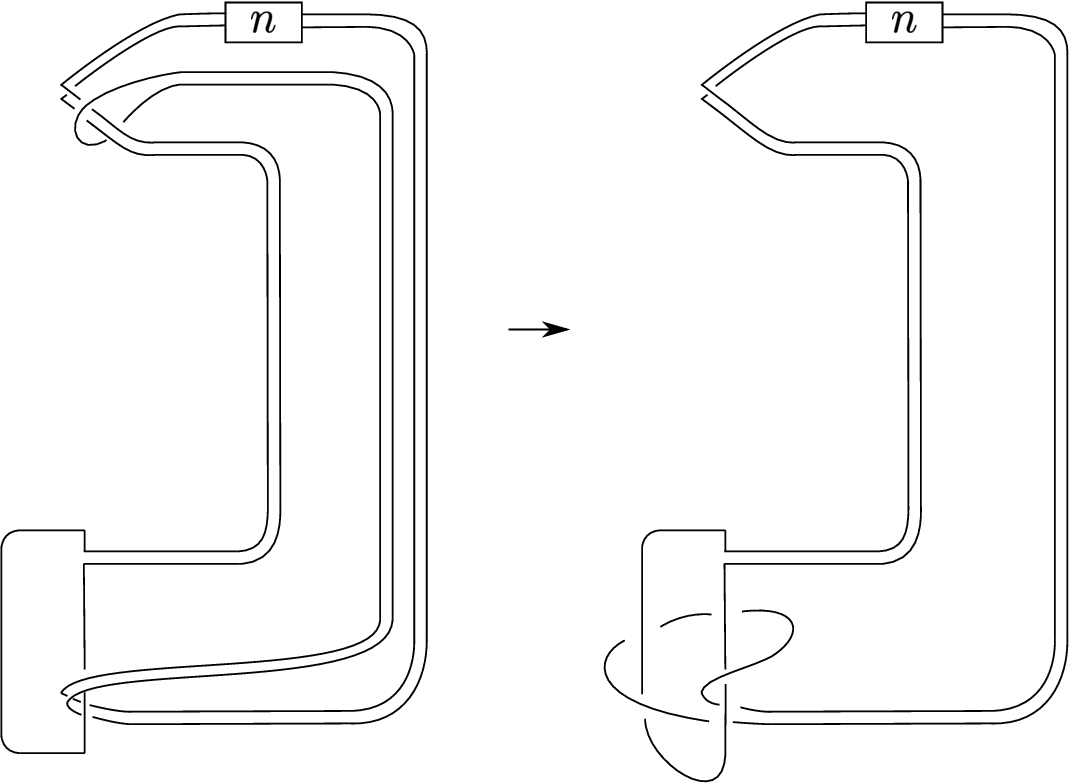}
\caption{A deformation of the knot $\ell_\infty$ (continued from Figure \ref{fig:k2mont2-inf}).
The right is the $(-3,3,n-1)$--pretzel knot.
}\label{fig:k2mont2-inf2}
\end{center}
\end{figure}

\begin{claim}\label{cl:ell0}
The double branched cover of $\ell_0$ is an L--space.
\end{claim}

\begin{proof}
For the crossing encircled in Figure \ref{fig:k2mont2-00},
we further perform the resolutions, which yield $\ell_{0\infty}$ and $\ell_{00}$.
Clearly, $\ell_{0\infty}=\ell_\infty$.
Hence $\mathrm{det}\, \ell_{0\infty}=9$.

\begin{figure}[htpb]
\begin{center}
\includegraphics[scale=0.4]{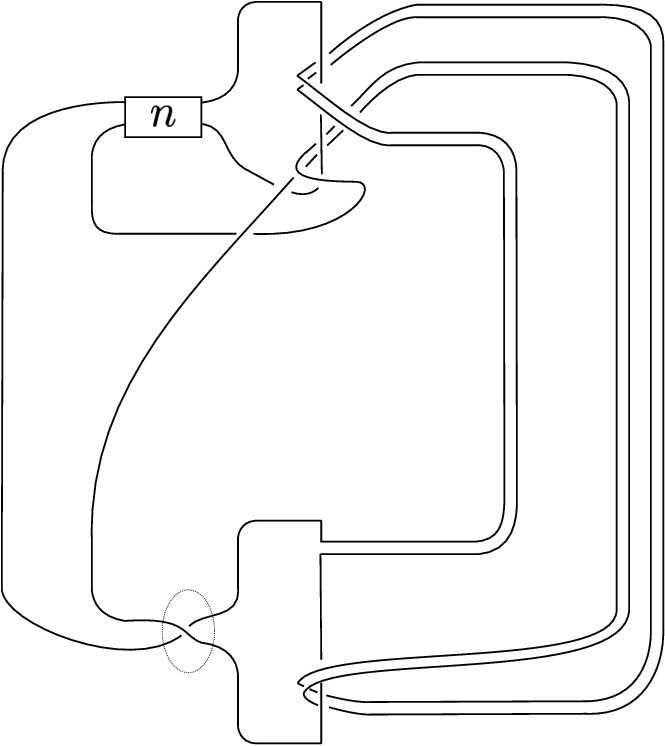}
\caption{The knot $\ell_0$.
}\label{fig:k2mont2-00}
\end{center}
\end{figure}

We can confirm that $\ell_{00}$ is the connected sum of the Hopf link and a Montesinos knot
as shown in Figures \ref{fig:k2mont2-00-1} and \ref{fig:k2mont2-00-2}.
From the diagram of Figure \ref{fig:k2mont2-00-2},
we see that $\mathrm{det}\,\ell_{00}=16n+2$.
Recall that $\mathrm{det}\,\ell_0=16n+11$.
Hence the equation $\mathrm{det}\,\ell_{0}=\mathrm{det}\,\ell_{0\infty}+\mathrm{det}\,\ell_{00}$ holds.

\begin{figure}[htpb]
\begin{center}
\includegraphics[scale=0.4]{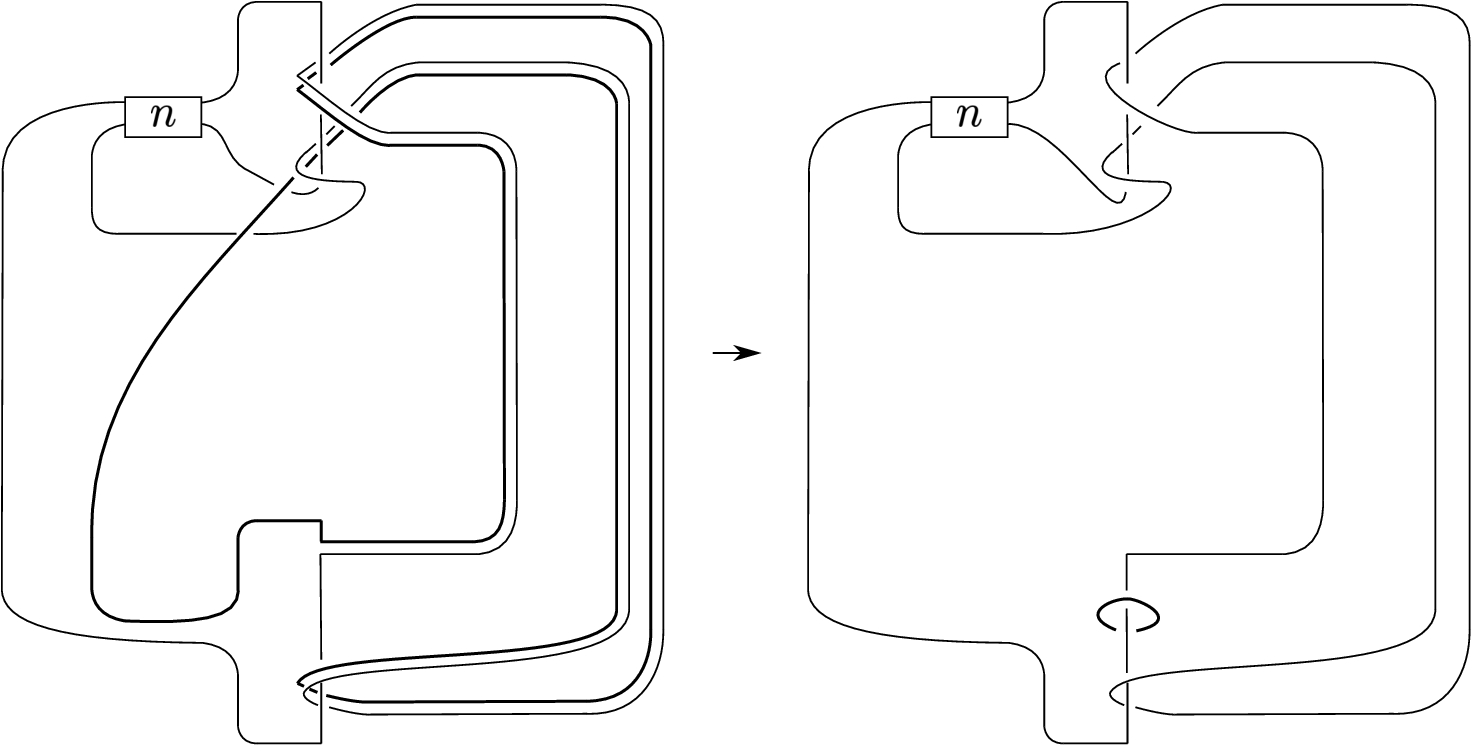}
\caption{The knot $\ell_{00}$ has the Hopf link as its connected summand.
}\label{fig:k2mont2-00-1}
\end{center}
\end{figure}

\begin{figure}[htpb]
\begin{center}
\includegraphics[scale=0.4]{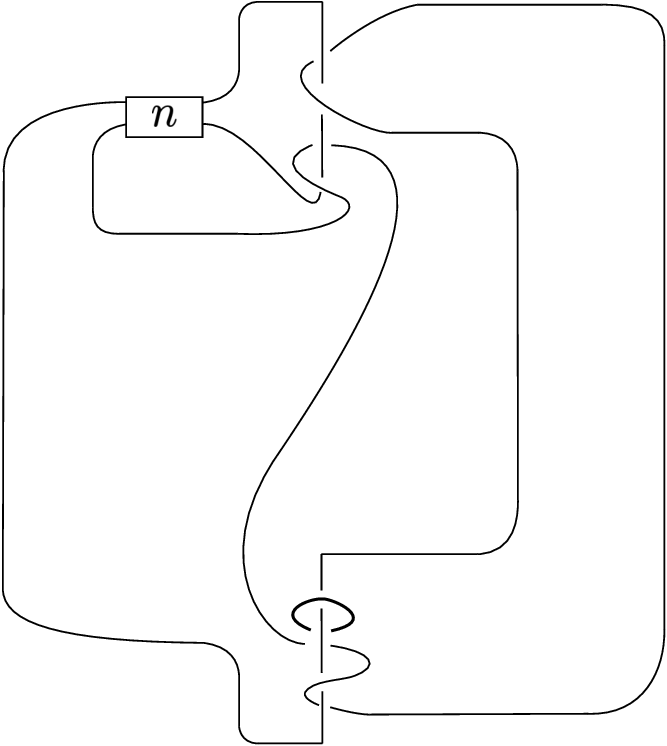}
\caption{The knot $\ell_{00}$ is the connected sum of the Hopf link
and the Montesinos knot $M(1/2,-1/3,n/(2n+1))$.
}\label{fig:k2mont2-00-2}
\end{center}
\end{figure}

From Claim \ref{cl:ellinfty}, the double branched cover of $\ell_{0\infty}$ is an L--space.
It remains to show that the double branched cover of $\ell_{00}$ is an L--space.

The double branched cover of the Montesinos knot 
$M=M(1/2,-1/3,n/(2n+1))$ is the Seifert fibered manifold $M(0;1/2,-1/3,n/(2n+1))$.
Since $M$ is homeomorphic to $M(-1;1/2,2/3,n/(2n+1))$,
set $r_1=2/3$, $r_2=1/2$ and $r_3=n/(2n+1)$.
Then $1\ge r_1\ge r_2\ge r_3\ge 0$.
We apply the criterion of \cite{LS} again.
If $1/m>r_3=n/(2n+1)\ge 1/3$, then $m<3$.
Hence there are no coprime integers $m>a>0$ such that $a/m>r_1=2/3$.
Thus $M$ is an L--space.

The double branched cover of $\ell_{00}$ is the connected sum of a lens space $L(2,1)$ and $M$.
Since the sum of L--spaces is an L--space \cite{OS}, we have the conclusion.
\end{proof}

By Claims \ref{cl:ellinfty} and \ref{cl:ell0}, we obtain that
the double branched cover of $\ell$ is an L--space.
\end{proof}

\section{Restorability of Alexander polynomials}\label{sec:restore}

In this section, we investigate the restorability of Alexander polynomial of an L--space knot
from the Upsilon invariant.

As easy examples,  we examine two torus knots.

\begin{example}
(1) Let $K=T(3,4)$.
Then $\Delta_K(t)=1-t+t^3-t^5+t^6$, so $\mathcal{S}_K=\{0,3,4\}\cup \mathbb{Z}_{\ge 6}$
and $\mathcal{G}_K=\mathbb{Z}_{<0}\cup \{1,2,5\}$.
It is easy to calculate $\Upsilon_K(t)$ as
\[
\Upsilon_K(t)=
\begin{cases}
-3t & \text{for $0\le t\le \frac{2}{3}$}, \\
-2 & \text{for $\frac{2}{3}\le t\le \frac{4}{3}$}, \\
3t-6 &  \text{for $\frac{4}{3}\le t\le 2$}.
\end{cases}
\]
The Legendre--Fenchel transformation on $\Upsilon_K(t)$ gives a function
\[
f(x)=
\begin{cases}
0 &  \text{for $x\le -3$}, \\
\frac{2}{3}(x+3) & \text{for $-3\le x\le 0$},\\
\frac{4}{3}x+2 & \text{for $0\le x\le 3$},\\
2x & \text{for $3\le x$}.
\end{cases}
\]
Of course, this is the convex hull of the gap function of $K$.
Figure \ref{fig:t34t35} shows the graphs of gap function of $K$ and $f$.

\begin{figure}[htpb]
\begin{center}
\includegraphics[scale=1.5]{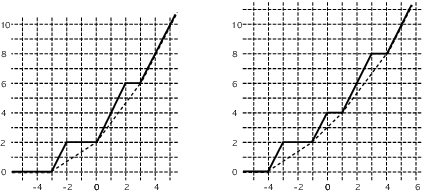}
\caption{The graphs of the gap functions and their convex hulls (broken line) of 
$T(3,4)$ (left) and $T(3,5)$ (right).
}\label{fig:t34t35}
\end{center}
\end{figure}

We consider the possibility of another gap function $G$ whose convex hull is $f$.
First, it forces $G(-3)=0$, $G(0)=2$ and $G(3)=6$.
Recall that each segment of the graph of a gap function has slope $0$ or $2$ as mentioned in Section \ref{sec:upsilon}.
Hence $G(-2)=2$.
Since a gap function is increasing, $G(-1)=2$.
Similarly, it is necessary that $G(1)=4$ and $G(2)=6$.
Thus $G$ coincides with the gap function of $K$.

This means that if another L--space knot $K'$ has the same Upsilon invariant as $K$, then
$\Delta_{K'}(t)=\Delta_K(t)$, because a gap function uniquely determines the Alexander polynomial.

(2) Let $K=T(3,5)$.
We have $\Delta_K(t)=1-t+t^3-t^4+t^5-t^7+t^8$, so
$\mathcal{S}_K=\{0,3,5,6\}\cup \mathbb{Z}_{\ge 8}$, and $\mathcal{G}_K=\mathbb{Z}_{<0}\cup \{1,2,4,7\}$.
Then $\Upsilon_K(t)$ is given as
\[
\Upsilon_K(t)=
\begin{cases}
-4t & \text{for $0\le t\le \frac{2}{3}$}, \\
-t-2 & \text{for $\frac{2}{3}\le t \le 1$}, \\
t-4 &  \text{for $1\le t \le \frac{4}{3}$},\\
4t-8 & \text{for $\frac{4}{3}\le t\le 2$}.
\end{cases}
\]
Figure \ref{fig:t34t35} shows the graphs of gap function of $K$ and the convex hull, which is the Legendre--Fenchel transform of $\Upsilon_K(t)$. 
As in (1), the convex hull uniquely restores the gap function.
\end{example}

In general, it is rare that the convex hull uniquely restores a gap function.
In Example \ref{ex:237}, we determined the gap function and its convex hull of the $(-2,3,7)$--pretzel knot
(see Figure \ref{fig:grid}).
It is possible that another gap function $G$ takes the same values on integers except $G(0)=6$, keeping the same convex hull.
This new gap function corresponds to 
the Alexander polynomial $\Delta(t)=1-t+t^3-t^5+t^7-t^9+t^{10}$.
This polynomial satisfies the condition of \cite{K}, but there is no hyperbolic L--space knot in Dunfield's list whose
Alexander polynomial is $\Delta(t)$.
It seems to be a hard question whether there exists a hyperbolic L--space knot with $\Delta(t)$.
Of course, there exists a hyperbolic knot whose Alexander polynomial is $\Delta(t)$ by \cite{Fu,St}.
Also, $\Delta(t)$ is the Alexander polynomial of the $(2,3)$--cable of $T(2,5)$, which is not an L--space knot \cite{Ho0}.

If we put off
the realizability of the Alexander polynomial or the gap function by a hyperbolic L--space knot, then
we can easily design many Alexander polynomials which are restorable from convex hulls.

It is a classical result that any polynomial $\Delta(t)$ satisfying $\Delta(1)=1$ and $\Delta(t^{-1})\stackrel{.}{=}\Delta(t)$
is realized by a knot in the $3$--sphere as its Alexander polynomial.
(Here, $\stackrel{.}{=}$ shows the equality up to units $\pm t^i$ in the Laurent polynomial ring $\mathbb{Z}[t,t^{-1}]$.)
Furthermore, we assume that $\Delta(t)$ has the form of (\ref{eq:alex}).
Formally, we define the formal semigroup $\mathcal{S}$ by (\ref{eq:ps}), and in turn, its gap set and the gap function.

\begin{proposition}\label{prop:restore}
Let $m\ge 3$ be an integer, and let $\Delta(t)=1-t+t^m-t^{m+1}+t^{m+2}-t^{2m+1}+t^{2m+2}$.
Then its gap function, defined formally, is uniquely determined from the convex hull.
\end{proposition}

Again, the polynomial $\Delta(t)$ in Proposition \ref{prop:restore}
satisfies the condition of \cite{K}, but it is open whether $\Delta(t)$ is realized by a hyperbolic L--space knot or not.
(When $m=3$, $\Delta(t)$ is the Alexander polynomial of $T(3,5)$.)

\begin{proof}
By (\ref{eq:ps}), the formal semigroup is $\mathcal{S}=\{0,m\}\cup \{m+2,m+3,\dots,2m\}\cup \mathbb{Z}_{2m+2}$,
so the gap set is $\mathcal{G}=\mathbb{Z}_{<0}\cup \{1,2,\dots,m-1\}\cup \{m+1,2m+1\}$.
Set $g=m+1$.
Then we can calculate the gap function $2J(-m)$ as in Table \ref{table:restore}.

\begin{table}[h]
\begin{tabular}{c|c|cccccc|ccc}
$m$    & $\le -m-1$ & $-m$ & $0$ & $2$ & $3$ & $\dots$ & $m$ & $m+2$ & $m+3$ & $\dots$ \\
\hline
$2J(-m)$ & 0  &2  & 4 &  $6$ &  $8$ &  $\dots$  & $2m+2$ & $2m+4$ & $2m+6$ &  $\dots$ 
\end{tabular}
\caption{The gap function $2J(-m)$.}\label{table:restore}
\end{table}

Let $f$ be the convex hull.  Then it is given by
\[
f(x)=
\begin{cases}
0 & \text{for $x\le -m-1$}, \\
\frac{2}{m}(x+m+1) & \text{for $-m-1\le x\le -1$}, \\
x+3 & \text{for $-1\le x\le 1$}, \\
\frac{2m-2}{m}(x-1)+4 & \text{for $1\le x\le m+1$}, \\
2x & \text{for $m+1\le x$}.
\end{cases}
\]
Since each segment of the graph of any gap function has slope $0$ or $2$,
there is no other gap function whose convex hull is $f$.
\end{proof}

Finally, we prove Theorem \ref{thm:restorable}.
 For reader's convenience, we record the braid words for the knots 
\texttt{t09847} and \texttt{v2871}.
Both are the closures of $4$--braids, whose words are almost the same:
\[
(\sigma_2\sigma_1\sigma_3\sigma_2)^3 (\sigma_2\sigma_1^2 \sigma_2) \sigma_1\quad  \text{and} \quad
(\sigma_2\sigma_1\sigma_3\sigma_2)^3 (\sigma_2\sigma_1^2 \sigma_2) \sigma_1^3.
\]

\begin{proof}[Proof of Theorem \ref{thm:restorable}]
Let $K$ be the hyperbolic knot \texttt{t09847} in the SnapPy census.
The Alexander polynomial is $\Delta_K(t)=1-t+t^4-t^5+t^7-t^9+t^{10}-t^{13}+t^{14}$, so
the formal semigroup is $\mathcal{S}_K=\{0,4,7,8,10,11,12\}\cup \mathbb{Z}_{\ge 14}$.

Figure \ref{fig:t09847} shows the graph of the gap function and its convex hull (we omit the details).
It consists of branches of types (a), (b), (c), (f), (g) and (h) of Figure \ref{fig:part} from the left.
Then there is no other gap function with the same convex hull.

\begin{figure}[htpb]
\begin{center}
\includegraphics[scale=1.3]{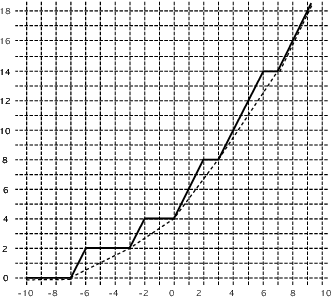}
\caption{The graph of the gap function and its convex hull (broken line) of \texttt{t09847}.
}\label{fig:t09847}
\end{center}
\end{figure}

Next, let $K$ be the hyperbolic knot \texttt{v2871}.
The Alexander polynomial is $1-t+t^4-t^5+t^7-t^8+t^9-t^{11}+t^{12}-t^{15}+t^{16}$, so
the formal semigroup is $\{0,4,7,9,10,12,13,14\}\cup \mathbb{Z}_{\ge 16}$ and
the gap set is $\mathbb{Z}_{<0}\cup \{1,2,3,5,6,8,11,15\}$.
Figure \ref{fig:v2871} shows the graph of the gap function and its convex hull.
In this case, the graph consists of branches of types
(a), (b), (c), (e), (f), (g) and (h) of Figure \ref{fig:part} from the left.
Again, there is no other gap function with the same convex hull.
\end{proof}

\begin{figure}[htpb]
\begin{center}
\includegraphics[scale=1.3]{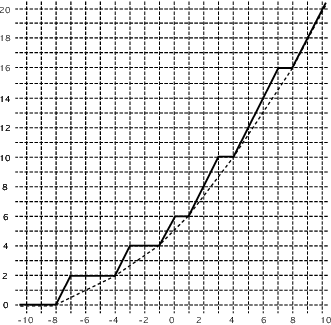}
\caption{The graph of the gap function and its convex hull (broken line) of \texttt{v2871}.
}\label{fig:v2871}
\end{center}
\end{figure}

\section*{Acknowledgement}

The author would like to thank Kouki Sato for valuable communication.

\bibliographystyle{amsplain}

\end{document}